%% file: main.tex
\documentclass[11pt,letterpaper]{article}

\usepackage[dvipsnames]{xcolor}

\input{header}

\title{Lehmer Codes and the Reverse-Complement Mapping from \underline{32}1-Avoiding Permutations to 3\underline{21}-Avoiding Permutations}
\author{Andrew Beveridge\footnote{Department of Mathematics, Statistics and Computer Science, Macalester College, St Paul, MN, USA}, Yufan Hu$^*$ and Yucheng Liu$^*$}

\date{}

\begin{document}

\maketitle


\begin{abstract}
Let $\mathcal{S}_n(\underline{32}1)$ and $\mathcal{S}_n(3\underline{21})$ denote the sets of $n$-permutations avoiding the vincular patterns $\underline{32}1$ and $3\underline{21}$, respectively. Using Lehmer codes, we realize these families as weighted posets $\mathcal{L}_n(\underline{32}1)$ and $\mathcal{L}_n(3\underline{21})$, where the weight of a code is the inversion number of its permutation. We show that the maximal elements of each of these posets, $\Max \mathcal{L}_n(\underline{32}1)$ and $\Max \mathcal{L}_n(3\underline{21})$, are enumerated by the Fibonacci numbers. We demonstrate that the classical reverse-complement map on permutations restricts to a natural bijection between these two sets of maximal elements, revealing a deep  symmetry between their underlying poset structures.




\end{abstract}


\medskip


\input{section-introduction}

\input{section-preliminaries}

\input{section-32-1}

\input{section-3-21}

\input{section-gravity-bijection}

\section{Conclusion}

\label{sec:conc}

The construction of $\Max \lehmerpat{n}{\underline{32}1}$ and $\Max_w \lehmerpat{n}{\underline{32}1}$ was established in \cite{BHR}, and this latter set corresponds to the $n$-permutations with maximum inversion number. 
Meanwhile, characterizing $\Max \lehmerpat{n}{\underline{32}1}$ and $\Max_w \lehmerpat{n}{\underline{32}1}$ remained open. Theorem \ref{thm:gravity_bijection} provides the explicit bijection resolving this gap. 
We show that the classical reverse-complement map $\pi \mapsto \revcompl{\pi}$ induces a simple mapping $\F$ from
$\Max \lehmerpat{n}{\underline{32}1}$ to $\Max \lehmerpat{n}{3\underline{21}}$ which has a natural geometric ``flip and fall'' interpretation. 

Looking to future work, we wonder what else can be said about the poset structures of $\lehmerpat{n}{\underline{32}1}$ and $\lehmerpat{n}{3\underline{21}}$, and whether there are further correspondences between them. It is known that $\lehmerpat{n}{\underline{32}1}$ is a meet-semlattice \cite{BHR}, but this is not the case for $\lehmerpat{n}{3\underline{21}}$. We are also curious about whether existing results about inversion numbers and 321-permutations can be adapted to $\underline{32}1$-avoiding permutations or $3\underline{21}$-avoiding permutations. This includes counting permutations by length and inversion number \cite{chen-mei-wang} and the inversion polynomial \cite{cheng_inversion_2013}.

\bibliography{biblio}

\end{document}

%% file: header.tex
\usepackage{amsmath, amssymb, amsfonts, latexsym, stmaryrd, enumerate, comment}
\usepackage{amsthm}
\usepackage{tikz, tikz-cd, tikzsymbols}
\usepackage[english]{babel}
\usepackage{xcolor}
\usepackage{mathdots}
\usepackage{mathtools}

\usetikzlibrary{decorations.pathreplacing}

\newtheorem{theorem}{Theorem}

\newtheorem{cor}[theorem]{Corollary}
\newtheorem{lemma}[theorem]{Lemma}

\newtheorem*{theorem*}{Theorem}
\newtheorem*{thm*}{Theorem}
\newtheorem*{lemma*}{Lemma}
\newtheorem*{conj*}{Conjecture}
\theoremstyle{definition}
\newtheorem{definition}[theorem]{Definition}

\numberwithin{theorem}{section}
\numberwithin{equation}{section}
\numberwithin{figure}{section}

\newcommand{\perm}[1]{\mathcal{S}_{#1}}
\newcommand{\permpat}[2]{\perm{#1}(#2)}

\newcommand{\lehmer}[1]{\mathcal{L}_{#1}}
\newcommand{\lehmerpat}[2]{\lehmer{#1}(#2)}

\newcommand{\desc}[1]{\mathcal{D}_{#1}}
\newcommand{\descpat}[2]{\desc{#1}(#2)}

\newcommand{\descmap}[1]{\Delta_{#1}}

\DeclareMathOperator{\inv}{inv}

\DeclareMathOperator{\Max}{Max}

\DeclareMathOperator{\Rev}{Rev}

\newcommand{\rev}[1]{#1^r}
\newcommand{\revlem}[1]{r(#1)}
\newcommand{\compl}[1]{#1^c}
\newcommand{\revcompl}[1]{#1^{r.c}}
\newcommand{\kompl}[1]{c(#1)}
\newcommand{\revkompl}[1]{c \circ \revlem{#1}}
\newcommand{\revkomplempty}{c \circ r}

\newcommand{\F}{F}
\newcommand{\flip}{\Phi}

\newcommand{\rmin}{\operatorname{RMin}}
\newcommand{\lmax}{\operatorname{LMax}}

\newcommand{\maxset}[2]{\Max \lehmerpat{#1}{#2}}

\newcommand{\comp}[1]{{#1}^c}

\newcounter{x}
\newcounter{y}

\newcommand\emptycol[1]{
   \draw[thick] ( - #1, 0) -- (-#1, #1) -- (-#1-1, #1) -- (-#1-1, 0) -- cycle;   
   \foreach \c in {1,...,#1} {
      \draw[thick] (-#1, \c) -- (-#1-1, \c);
  }      
}

\newcommand\fillcol[3]{
    \fill[#3!50] (-#1, 0) -- (-#1, #2) -- (-#1-1, #2) -- (-#1-1, 0) -- cycle;
    
}

\newcommand\fillrow[3]{
    \fill[#3!50] (0, #1) -- (#2, #1) -- (#2, #1+1) -- (0, #1+1) -- cycle;
}

\newcommand\newtri[3][(0,0)]{
 \setcounter{y}{0}
  \foreach \a in {#2} {
      \addtocounter{y}{1}
    }

\begin{scope}[shift={#1}]
\begin{scope}[scale=0.33, shift={(-\value{y}/2,0)}]    
 \setcounter{x}{0}
  \foreach \a in {#2} {
      \addtocounter{x}{-1}
      \fillcol{\value{x}}{\a}{#3}
    }
  \end{scope}
\end{scope}

\begin{scope}[shift={#1}]
\begin{scope}[scale=0.33, shift={(\value{y}/2+1,0)}]    
 \setcounter{x}{0}
  \foreach \a in {#2} {
      \addtocounter{x}{1}
      \emptycol{\value{x}}
    }
  \end{scope}
\end{scope}
}

\newcommand\fliptri[3][(0,0)]{
 \setcounter{y}{0}
  \foreach \a in {#2} {
      \addtocounter{y}{1}
    }

\begin{scope}[shift={#1}]
\begin{scope}[scale=0.33, shift={(-\value{y}/2,0)}]    
 \setcounter{x}{-1}
  \foreach \a in {#2} {
      \addtocounter{x}{1}
      \fillrow{\value{x}}{\a}{#3}
    }
  \end{scope}
\end{scope}

\begin{scope}[shift={#1}]
\begin{scope}[scale=0.33, shift={(\value{y}/2+1,0)}]    
 \setcounter{x}{0}
  \foreach \a in {#2} {
      \addtocounter{x}{1}
      \emptycol{\value{x}}
    }
  \end{scope}
\end{scope}
}

\usepackage{hyperref}

\hypersetup{
  colorlinks   = true,
  urlcolor     = blue, 
  linkcolor    = blue,
  citecolor   = red
}

%% file: section-introduction.tex
\section{Introduction}

The study of pattern avoidance in permutations has evolved into a dynamic branch of enumerative combinatorics since the first systematic investigation by Simion and Schmidt \cite{simion1985restricted}.
Given permutations $\tau = (\tau_1, \tau_2, \ldots, \tau_k) \in \perm{k}$ and $\pi = (\pi_1, \pi_2, \cdots, \pi_n) \in \perm{n}$, we say that $\pi$ \emph{contains} $\tau$ when there exist indices $1 \leq i_1 < i_2 < \cdots < i_k \leq n$ such that the entries of subsequence $(\pi_{i_1}, \pi_{i_2}, \ldots, \pi_{i_k})$ are in the same relative order as the entries of $\tau$; otherwise, we say that $\pi$ \emph{avoids} $\tau$. 
We use $\permpat{n}{\tau}$ to denote the subset $\tau$-avoiding permutations in $\perm{n}$.
Investigation of \emph{vincular patterns} (also called \emph{generalized patterns})
was pioneered by Babson and Steingr\'imsson \cite{Babson2000}. A \emph{vincular pattern} $\tau$ is a permutation in $\perm{k}$, some of whose consecutive entries are underlined. If $\pi \in \perm{n}$ contains vincular pattern $\tau$, and $\tau$ contains $\underline{\tau_i \tau_{i+1} \cdots \tau_j}$, then the entries of $\pi$ corresponding to $\tau_i, \tau_{i+1}, \ldots, \tau_j$ in $\pi$ must be adjacent.
See Kitaev \cite{kitaev} for a comprehensive treatment of patterns in permutations, and see Steingr\'imsson \cite{steingrimsson} for a survey on vincular patterns.

Often, the fundamental objective of pattern avoidance research is to enumerate the set 
or to establish structural bijections between seemingly distinct pattern avoiding families; examples include \cite{claesson, biers-ariel, li2022, mansour2025-flatten}.  In this paper, we explore the structure of 
$\permpat{n}{\underline{32}1}$ and $\permpat{n}{3\underline{21}}$.
Claesson \cite{claesson} proved that both families are enumerated by the Bell numbers,
$| \permpat{n}{\underline{32}1} | = | \permpat{n}{3\underline{21}} | = B_n.$
By translating these permutations into their \emph{Lehmer codes}, 
we uncover their underlying poset structures and characterize their maximal configurations, providing both exact enumerative results and natural bijective mappings.

\begin{definition}
    The collection $\lehmer{n}$ of \emph{Lehmer codes} of length $n$ is 
    $$
    \lehmer{n} = \{ (p_1, p_2, \ldots, p_n) : 0 \leq p_i \leq n-i \text{ for } 1 \leq i \leq n \}.
    $$
    For a permutation $\pi \in \perm{n}$, its Lehmer code is
$$
L(\pi) := ( p_1, p_2,\ldots , p_n) \quad \mbox{where} \quad
p_i = \left| \left\{ j > i : \pi_j < \pi_i \right\} \right|.
$$     
Given a vincular pattern $\tau$, we define
$$
\lehmerpat{n}{\tau} = \{ L(\pi) : \pi \in \permpat{n}{\tau}\}
$$
to be the Lehmer codes corresponding to the $\tau$-avoiding permutations.
The \emph{weight} of a Lehmer code is $w(p) = \sum_{i=1}^{n} p_i$, and it is clear that the inversion number of $\pi$ is $\inv (\pi) = w(p)$.
The mapping $L : \perm{n} \rightarrow \lehmer{n}$  is a bijection \cite{lehmer}, and we describe the inverse mapping in Section \ref{sec:decode}.
\end{definition}

The collection of Lehmer codes admits a natural poset structure $(\lehmer{n}, \preceq)$. Given $s,t \in \lehmer{n}$, we have $s \preceq t$ when $s_i \leq t_i$ for $1 \leq i \leq n$. For explorations of this poset structure, see \cite{denoncourt, tomie, bouvel}. Given a vincular pattern $\tau$, we define the subposet  $(\lehmerpat{n}{\tau}, \preceq)$ where $$
\lehmerpat{n}{\tau} = \{ L(\pi) : \pi \in \permpat{n}{\tau}\}.
$$
We denote its maximal elements by $\Max \lehmerpat{n}{\tau}$ and its maximal weight elements by $\Max_w \lehmerpat{n}{\tau}$. Note that this second set 
corresponds to the subset of $\permpat{n}{\underline{32}1}$ of permutations with maximum inversion number.
Figure \ref{fig:32-1-3-21-posets} shows the posets $\lehmerpat{4}{\underline{32}1}$ and
$\lehmerpat{4}{3\underline{21}}$ using a geometric representation of a Lehmer code. We display $p \in \lehmer{n}$ as columns of (filled) boxes in a triangular grid of size $n-1$, where column $i$ contains $0 \leq p_i \leq n-i$ boxes.

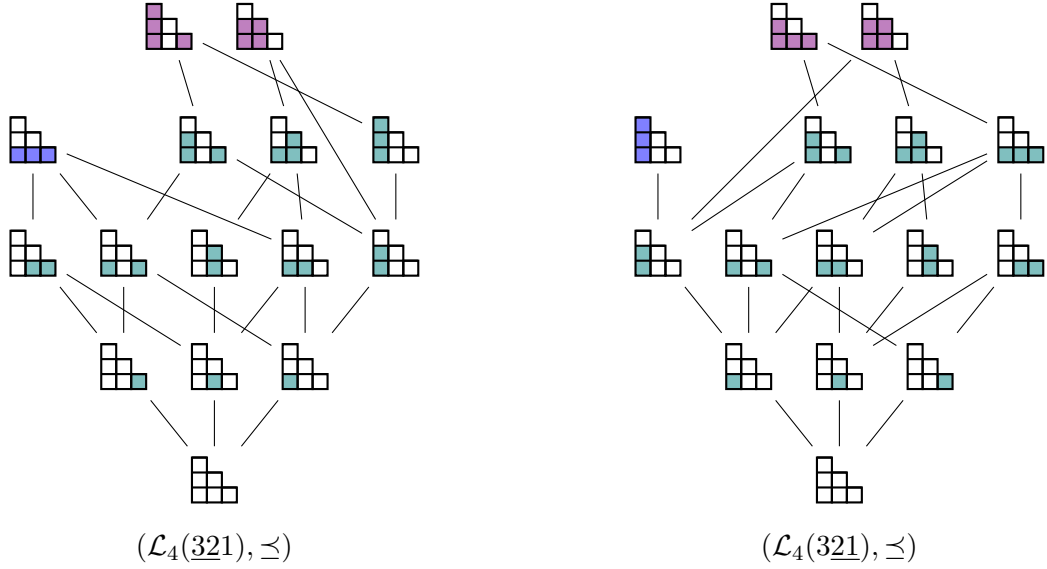
\begin{figure}[ht]
\centering

\begin{minipage}[t]{0.48\textwidth}
\centering
\begin{tikzpicture}[scale=0.6]
\node (000) at (0,0) { \begin{tikzpicture}[scale=0.6]\newtri[(-1,0)]{0,0,0}{teal} \end{tikzpicture}};

\node (010) at (0,2.5) {\begin{tikzpicture}[scale=0.6]\newtri[(-1,0)]{0,1,0}{teal}\end{tikzpicture}};
\node (100) at (2,2.5) {\begin{tikzpicture}[scale=0.6]\newtri[(-1,0)]{1,0,0}{teal}\end{tikzpicture}};
\node (001) at (-2,2.5) {\begin{tikzpicture}[scale=0.6]\newtri[(-1,0)]{0,0,1}{teal}\end{tikzpicture}};

\node (020) at (0,5) {\begin{tikzpicture}[scale=0.6]\newtri[(-1,0)]{0,2,0}{teal}\end{tikzpicture}};
\node (110) at (2,5) {\begin{tikzpicture}[scale=0.6]\newtri[(-1,0)]{1,1,0}{teal}\end{tikzpicture}};
\node (011) at (-4,5) {\begin{tikzpicture}[scale=0.6]\newtri[(-1,0)]{0,1,1}{teal}\end{tikzpicture}};
\node (101) at (-2,5) {\begin{tikzpicture}[scale=0.6]\newtri[(-1,0)]{1,0,1}{teal}\end{tikzpicture}};
\node (200) at (4,5) {\begin{tikzpicture}[scale=0.6]\newtri[(-1,0)]{2,0,0}{teal}\end{tikzpicture}};

\node (120) at (1.75,7.5) {\begin{tikzpicture}[scale=0.6]\newtri[(-1,0)]{1,2,0}{teal}\end{tikzpicture}};
\node (111) at (-4,7.5) {\begin{tikzpicture}[scale=0.6]\newtri[(-1,0)]{1,1,1}{blue}\end{tikzpicture}};
\node (201) at (-0.25,7.5) {\begin{tikzpicture}[scale=0.6]\newtri[(-1,0)]{2,0,1}{teal}\end{tikzpicture}};
\node (300) at (4,7.5) {\begin{tikzpicture}[scale=0.6]\newtri[(-1,0)]{3,0,0}{teal}\end{tikzpicture}};

\node (220) at (1,10) {\begin{tikzpicture}[scale=0.6]\newtri[(-1,0)]{2,2,0}{violet}\end{tikzpicture}};
\node (301) at (-1,10) {\begin{tikzpicture}[scale=0.6]\newtri[(-1,0)]{3,0,1}{violet}\end{tikzpicture}};

\foreach \from/\to in {
000/010, 000/100, 000/001,
010/020, 010/110, 010/011,
100/101, 100/200, 100/110,
001/011, 001/101,
020/120, 200/220,
110/111, 110/120,
011/111,
101/111, 101/201,
200/201, 200/300,
201/301, 300/301,
120/220}
\draw (\from)--(\to);

\node at (0,-1.5) {$\left(\lehmerpat{4}{\underline{32}1},\preceq\right)$};

\end{tikzpicture}

\end{minipage}
\hfill
\begin{minipage}[t]{0.48\textwidth}
\centering
\begin{tikzpicture}[scale=0.6]

\node (000) at (0,0) { \begin{tikzpicture}[scale=0.6]\newtri[(-1,0)]{0,0,0}{teal} \end{tikzpicture}};

\node (010) at (0,2.5) {\begin{tikzpicture}[scale=0.6]\newtri[(-1,0)]{0,1,0}{teal}\end{tikzpicture}};
\node (001) at (2,2.5) {\begin{tikzpicture}[scale=0.6]\newtri[(-1,0)]{0,0,1}{teal}\end{tikzpicture}};
\node (100) at (-2,2.5) {\begin{tikzpicture}[scale=0.6]\newtri[(-1,0)]{1,0,0}{teal}\end{tikzpicture}};

\node (110) at (0,5) {\begin{tikzpicture}[scale=0.6]\newtri[(-1,0)]{1,1,0}{teal}\end{tikzpicture}};
\node (020) at (2,5) {\begin{tikzpicture}[scale=0.6]\newtri[(-1,0)]{0,2,0}{teal}\end{tikzpicture}};
\node (200) at (-4,5) {\begin{tikzpicture}[scale=0.6]\newtri[(-1,0)]{2,0,0}{teal}\end{tikzpicture}};
\node (101) at (-2,5) {\begin{tikzpicture}[scale=0.6]\newtri[(-1,0)]{1,0,1}{teal}\end{tikzpicture}};
\node (011) at (4,5) {\begin{tikzpicture}[scale=0.6]\newtri[(-1,0)]{0,1,1}{teal}\end{tikzpicture}};

\node (120) at (1.75,7.5) {\begin{tikzpicture}[scale=0.6]\newtri[(-1,0)]{1,2,0}{teal}\end{tikzpicture}};
\node (300) at (-4,7.5) {\begin{tikzpicture}[scale=0.6]\newtri[(-1,0)]{3,0,0}{blue}\end{tikzpicture}};
\node (201) at (-.25,7.5) {\begin{tikzpicture}[scale=0.6]\newtri[(-1,0)]{2,0,1}{teal}\end{tikzpicture}};
\node (111) at (4,7.5) {\begin{tikzpicture}[scale=0.6]\newtri[(-1,0)]{1,1,1}{teal}\end{tikzpicture}};

\node (220) at (1,10) {\begin{tikzpicture}[scale=0.6]\newtri[(-1,0)]{2,2,0}{violet}\end{tikzpicture}};
\node (211) at (-1,10) {\begin{tikzpicture}[scale=0.6]\newtri[(-1,0)]{2,1,1}{violet}\end{tikzpicture}};

\foreach \from/\to in {
000/010, 000/001, 000/100,
010/110, 010/020, 010/011,
001/011, 001/101,
100/110, 100/101, 100/200,
110/120, 110/111,
020/120, 200/220,
200/300, 200/201,
101/201, 101/111,
011/111,
120/220,
201/211, 111/211}
\draw (\from)--(\to);

\node at (0,-1.5) {$\left(\lehmerpat{4}{3\underline{21}},\preceq\right)$};

\end{tikzpicture}

\end{minipage}

\caption{The posets of $\underline{32}1$-avoiding and $3\underline{21}$-avoiding Lehmer codes of length $4$. Each poset has two maximum weight Lehmer codes (colored violet) and three maximal Lehmer codes (colored violet and blue). The codes are arranged so that the elements of $\lehmerpat{4}{\underline{32}1}$ correspond to the elements of 
$\lehmerpat{4}{3\underline{21}}$ via the reverse-complement mapping $\revkomplempty$.
}
\label{fig:32-1-3-21-posets}
\end{figure}

Both $\Max \lehmerpat{n}{\underline{32}1}$ and $\Max_w \lehmerpat{n}{\underline{32}1}$ were characterized in \cite{BHR}. The first is a Fibonacci family (with a simple recursive construction), and the second family is enumerated by OEIS A209561 \cite{oeis}.
They also provide a method to generate these maximum weight elements.
As a corollary to their main result, they note that taking the reverse-complement of a $\underline{32}1$-avoiding permutation yields a corresponding $3\underline{21}$-avoiding permutation while also preserving the inversion number, and therefore $|\Max_w \lehmerpat{n}{3\underline{21}}| = |\Max_w \lehmerpat{n}{\underline{32}1}|$. However, they do not characterize or construct either the maximal elements $\Max \lehmerpat{n}{3\underline{21}}$ or the maximum weight elements of $\Max_w \lehmerpat{n}{3\underline{21}}$. We address this gap by providing insight into the reverse-complement mapping between $\maxset{n}{\underline{32}1}$ and
$\maxset{n}{3\underline{21}}$.

For permutation $\pi = \pi_1 \pi_2 \cdots \pi_n$, its \emph{reverse} is $\rev{\pi} = \pi_n \pi_{n-1} \cdots \pi_1$ and its \emph{complement} is
$\comp{\pi}= \pi_1'\pi_2'\cdots\pi_n'$ where $\pi_k' = n+1-\pi_k$. The \emph{reverse-complement} of $\pi$ is $\revcompl{\pi} := \compl{(\rev{\pi})}$. These three mappings induce mappings between Lehmer codes. When permutation $\pi$ has Lehmer code $p$, we define $\revlem{p} = L(\rev{\pi})$ and $\kompl{p} = L(\compl{\pi})$ and $\revkompl{p} = L(\revcompl{\pi})$. In general,  $\revkomplempty$ is cumbersome to describe, requiring knowledge of both $p$ and $\pi$. Our main contribution is to show that the reverse-complement mapping $\pi \mapsto \revcompl{\pi}$ induces a natural bijection between $\Max \lehmerpat{n}{\underline{32}1}$ and $\Max \lehmerpat{n}{3\underline{21}}$ that is straightforward to calculate.
The key is to focus on the descent sets of these maximum weight Lehmer codes, so we introduce a bit more notation.

\begin{definition}
\label{def:descent-set}
    Let $p=(p_1,\ldots,p_n)\in\lehmer{n}$. The \emph{descent set} of $p$ is
    \[
        D(p)= \{j\in[n-1]:p_j>p_{j+1}\}.
    \]
    When $|D(p)|=m$, we always write
    \[
     D(p) = \{d_1,d_2,\ldots,d_m\}
    \]
    where the indices $d_1 < d_2 < \cdots < d_m$ are listed in increasing order.
\end{definition}

\begin{definition}
\label{def:descent-set-map}
Given vincular pattern $\tau$, we define the collection of descent sets
$$
\descpat{n}{\tau} = \{ D(p) : p \in \Max \lehmerpat{n}{\tau} \}
$$
and the mapping
$$
\descmap{\tau} : \Max \lehmerpat{n}{\tau} \rightarrow \descpat{n}{\tau}  
$$
given by $\descmap{\tau}(p) = D(p)$.
\end{definition}

We will see that the descent sets 
$\descpat{n}{\underline{32}1}$ and $\descpat{n}{3\underline{21}}$ are easy to describe: each one is a simple family that adheres to a Fibonacci recurrence. Furthermore, we will show that the mappings $\descmap{\underline{32}1}$ and $\descmap{3\underline{21}}$ are bijections. This brings us to our main result.



\begin{theorem}
\label{thm:gravity_bijection}
Let $\F: \maxset{n}{\underline{32}1} \rightarrow \maxset{n}{3\underline{21}}$
be given by
\begin{equation}
\label{eqn:gravity_bijection}
 \F := \descmap{3\underline{21}}^{-1} \circ \flip \circ \descmap{\underline{32}1}
\end{equation}
where 
\begin{equation}
\label{eqn:psi}
\flip(D) = \{ n-d : d \in D\}.   
\end{equation}
Then the mapping $\F$ is a bijection. Furthermore, $\F$ corresponds to taking the reverse-complement of the associated permutation. That is, for every $p \in \maxset{n}{\underline{32}1}$, we have
$$
\F(p) = \revkompl{p}.
$$
As a consequence, $w(p)= w(\F(p))$.
\end{theorem}


For example, consider the permutation $\pi = 789156243 \in \permpat{9}{\underline{32}1}.$ 
Its reverse-complement is $\sigma := \revcompl{\pi} = \compl{(342651987)} = 768459123$.
Their Lehmer codes are $p:= L(\pi) = (6,6,6,0,3,3,0,1,0)$ and $s:= L(\sigma)=(6,5,5,3,3,3,0,0,0)$.
Lemmas \ref{lem:maximal_32_1_lehmer_condition} and \ref{lem:maximal_condition_3_21} below respectively show that $p \in \maxset{9}{\underline{32}1}$ and $s \in \maxset{9}{3\underline{21}}$. The  descent sets are $D(p)=\{3,6,8\}$ and $D(s)=\{1,3,6\}$. Theorem \ref{thm:gravity_bijection} holds for this example because $\flip(D(p)) = \{ 9-3, 9-6,9-8\} = \{1,3,6\} = D(s)$.

Furthermore, the bijection $\F$ between maximal codes has a natural geometric interpretation. Figure \ref{fig:reflect-gravity} shows the mapping, along with an intermediate configuration. Starting with $p \in \maxset{n}{\underline{32}1}$, we reflect through the line $x=y$. We then allow gravity to act on the floating boxes, resulting in the image $\F(p) \in \maxset{n}{3\underline{21}}$. Crucially, this ``flip and fall'' interpretation holds for the maximal codes, but does not hold in general: Figure \ref{fig:32-1-3-21-posets} shows that most pairs related by the reverse-complement map do not adhere to this geometric translation. 



\begin{figure}
 
\begin{center}
\begin{tikzpicture}[scale=0.7]

\newtri[(0,0)]{6,6,6,0,3,3,0,1}{teal}

\node at (0,-0.75) {(a)};

\node at (2.5,1.5) {$\rightarrow$};

\fliptri[(5,0)]{6,6,6,0,3,3,0,1}{teal}

\node at (5,-0.75) {(b)};

\node at (7.5,1.5) {$\rightarrow$};

\fliptri[(10,0)]{6,6,6,3,3,1,0,0}{teal}

\node at (10,-0.75) {(c)};

    
\end{tikzpicture}

\end{center}

\caption{The geometric ``flip and fall'' interpretation of the mapping $\F$. (a) Start with the triangle representation of $p \in \maxset{9}{\underline{32}1}$. (b) Reflect through the line $y=x$. (c) Allow floating boxes to fall down to obtain $\F(p) \in \maxset{9}{3\underline{21}}$.}

\label{fig:reflect-gravity}
   
\end{figure}
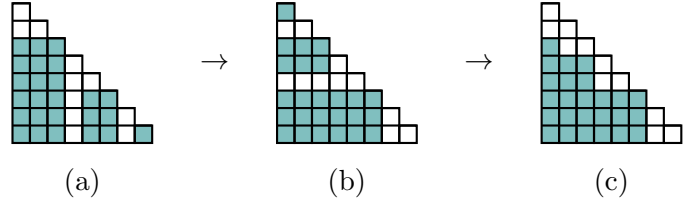





We close this section with a couple of reflective comments. 
First, because the reverse-complement mapping preserves the inversion number, it is clear that $\revkomplempty$ restricts to a mapping between maximum weight Lehmer codes $\Max_w \lehmerpat{n}{\underline{32}1}$ and $\Max_w \lehmerpat{n}{3\underline{21}}$. However, it is quite satisfying (and perhaps surprising) that $\revkomplempty$ restricts to a mapping between maximal Lehmer codes $\Max \lehmerpat{n}{\underline{32}1}$ and $\Max \lehmerpat{n}{3\underline{21}}$. This reveals a deep structural symmetry between the poset structures of $\lehmerpat{n}{\underline{32}1}$ and $\lehmerpat{n}{3\underline{21}}$.
Second, we note that Theorem 4.3 of \cite{BHR} gives a straightforward method that constructs the maximum weight codes of $\Max_w \lehmerpat{n}{\underline{32}1}$. We can now apply $\F$ to these Lehmer codes to construct the maximum weight codes of $\Max_w \lehmerpat{n}{3\underline{21}}$, resolving an open question from \cite{BHR}.

\subsection{Roadmap}

The paper is structured as follows. Section \ref{sec:prelim} contains some general results about Lehmer codes and permutations. Section \ref{sec:32-1-permutations} develops our understanding of the maximal Lehmer codes $\lehmerpat{n}{\underline{32}1}$ and their descent sets $\descpat{n}{\underline{32}1}$. Section \ref{sec:3-21-permutations} develops analogous results for ${\underline{32}1}$-avoiding permutations. Section \ref{sec:gravity-bijection} contains the proof of Theorem \ref{thm:gravity_bijection}, and we reflect on our results in Section \ref{sec:conc}.

We assume $n \geq 2$ throughout the paper, to avoid trivial arguments about $\perm{1}$.

%% file: section-preliminaries.tex
\section{Preliminaries}
\label{sec:prelim}

We prove some quick results about Lehmer codes for general permutations. 
First, we describe the Lehmer codes for the complement of a permutation and the reverse of a permutation. Then we make some observations about the left-to-right maxima of $\pi$ and the right-to-left minima of $\pi$. Along the way, we define the zero set of a Lehmer code.

\subsection{Decoding a Lehmer Code}

\label{sec:decode}

Given a Lehmer code $p \in \lehmer{n}$, we explain how to recover its corresponding permutation $\pi \in \perm{n}$. We start by creating a candidate list $C=(1,2,\ldots, n)$. We then process the entries of $p=(p_1, \ldots, p_n)$ from left to right, updating $C$ as we go. At step $i$, we skip over the $p_i$ smallest elements of $C$, setting $\pi_i$ to be the $(p_i+1)$ entry of $C$, which we also remove from list $C$. This guarantees that there will be exactly $p_i$ elements smaller than $\pi_i$ among the entries $i+1 \leq j \leq n$.

For example, suppose that we want to decode $p=(2,0,1,0) \in \lehmer{4}$. We start with $C=(1,2,3,4)$. We set $\pi_1=3$ (skipping over $p_1=2$ entries of $C$), and update $C=(1,2,4)$. Next, we set $\pi_2=1$ (skipping over $p_2=0$ entries of $C$), and update $C=(2,4)$. We then set $\pi_3=4$ (skipping over $p_3=1$ entries of $C$), and update $C=(2)$. Finally, we set $\pi_4=2$. So the corresponding permutation is $\pi=3142 \in \perm{4}$.

\subsection{Complement Permutation and Reverse Permutation}

Let permutation $\pi \in \perm{n}$ have Lehmer code $p \in \lehmer{n}$. We give formulas for the Lehmer codes for $\revlem{p} = L(\rev{\pi})$ and $\kompl{p} = L(\comp{\pi})$. 

 The Lehmer code $\kompl{p}$ for the complement $\comp{\pi}$ is simple to calculate.

\begin{lemma}
    \label{lem:complement_of_lehmer}
    Let $p \in \lehmer{n}$ be the Lehmer code of $\pi \in \perm{n}$, then the Lehmer code of $\compl{\pi}$ is
        $$
        \kompl{p}_i=(n-i)-p_i.
        $$
\end{lemma}

\begin{proof}
    Let $q=L(\compl{\pi})$. By the definition of Lehmer code, for each $1\le i\le n$,
    \[
        p_i=|\{j>i:\pi_j < \pi_i\}|.
    \]
    Since $\compl{\pi}_i=n+1-\pi_i$, we have
    \[
        q_i = |\{j>i:\compl{\pi}_j  < \compl{\pi}_i \}| = |\{j>i:\pi_j  > \pi_i \}|.
    \]
    Among the $n-i$ indices $j>i$, exactly $p_i$ of them satisfy $\pi_j <\pi_i$. Since $\pi$ is a permutation, no equality can occur. Therefore the remaining $(n-i)-p_i$ indices satisfy $\pi_j >\pi_i$. Hence
    $
    q_i=(n-i)-p_i
    $
    for every $1\le i\le n$. Thus $q=\kompl{p}$, as desired.
\end{proof}


For example, when $\pi=52718364$ and $p=(4,1,4,0,3,0,1,0)$, 
we have $\compl{\pi}=47281635$, which has Lehmer code $L(\compl{\pi})=(3,5,1,4,0,2,0,0)$.
This matches the formula for $\kompl{p}$ in Lemma \ref{lem:complement_of_lehmer}.

The Lehmer code $\revlem{p}$ for the reverse $\rev{\pi}$ is more complicated: we require both $\pi$ and $p$.

\begin{lemma}
    \label{lem:reversal_of_lehmer}
    Let $p \in \lehmer{n}$ be the Lehmer code of $\pi \in \perm{n}$, then the Lehmer code of $\rev{\pi}$ is
        $$
        \revlem{p}_i = \pi_{n + 1-i} -1 -p_{n + 1-i}.
        $$
\end{lemma}

\begin{proof}
    Let $q=L(\rev{\pi})$. Fix $1\le i\le n$ and set $k=n+1-i$.
    Then
    \[
        q_i = |\{j>i:\rev{\pi}_j < \rev{\pi}_i\}| = |\{ \ell<k:\pi_{\ell} <\pi_k \}|.
    \]
    Thus $q_i$ counts the entries smaller than $\pi_k$ that appear to the left of position $k$ in $\pi$. Since there are $\pi_k-1$ entries smaller than $\pi_k$ in total, and since
    \[
        p_k=|\{ \ell > k:\pi_{\ell} < \pi_k \}|
    \]
    counts those appearing to the right of position $k$, we obtain
    \[
        q_i=\pi_k-1-p_k.
    \]
    This proves the claim.
\end{proof}


For example, when $\pi=52718364$ and $p=(4,1,4,0,3,0,1,0)$, we have
$\rev{\pi} = 46381725$, which has Lehmer code $L(\rev{\pi})=(3,4,2,4,0,2,0,0)$.
This matches the formula for $\revlem{p}$ in Lemma \ref{lem:reversal_of_lehmer}.

We note that since $\revlem{p}$ requires the explicit entries of $\pi$, so does $\revkompl{p}$. Our Theorem \ref{thm:gravity_bijection} shows that when restricting to maximal Lehmer codes, we have $\revkompl{p} = F(p)$, and we can cleanly calculate $F(p)$ directly from $p$ alone.

\subsection{Left-to-Right Maxima and Right-to-Left Minima}

We now turn to the left-to-right maxima and the right-to-left minima of $\pi$,  making observations about how they manifest in the Lehmer code for $\pi$. 

\begin{definition}
    The left-to-right maxima of $\pi$ is the index set
    $$
    \lmax(\pi) = \{ i \in [n] : \pi_i > \pi_j \mbox{ for } 1 \leq j < i \}.
    $$
\end{definition}

\begin{lemma}
\label{lem:lehmer_equation_lmax}
    Let $\pi = (\pi_1,\pi_2,\cdots,\pi_n) \in \perm{n}$ with Lehmer code $L(\pi) = p = (p_1,p_2,\cdots, p_n)$. If $i \in \lmax(\pi)$, then
        \[
        p_i + i = \pi_i.
        \]
\end{lemma}

\begin{proof}
Since $i \in \lmax(\pi)$, all $(i-1)$ entries before position $i$ are smaller than $\pi_i$. 
Therefore, the
number of entries smaller than $\pi_i$ that appear after position $i$ is
\[
p_i = (\pi_i-1)-(i-1)=\pi_i-i.
\]
Equivalently, $p_i+i = \pi_i$.
\end{proof}


For example, when $\pi=52718364\in\perm{8}$ and $p=(4,1,4,0,3,0,1,0)$, the left-to-right maxima of $\pi$ occur at positions $\lmax(\pi) = \{1, 3, 5\}$. For these positions, we have
$p_1+1 = 5 = \pi_1$ and $p_3+3 7 = \pi_3$ and $p_5+5 = 8 = \pi_5$, as promised by 
 Lemma \ref{lem:lehmer_equation_lmax}.

\begin{definition}
The right-to-left minima of $\pi \in \perm{n}$ is the index set
$$
\rmin(\pi) = \{ i: \pi_i < \pi_j \mbox{ for } i < j \leq n \}.
$$
\end{definition}

\begin{definition}
\label{def:zero-set}
    Let $p=(p_1,\ldots,p_n)\in\lehmer{n}$. The \emph{zero set} of $p$ is
    \[
        Z(p)= \{j\in[n]:p_j = 0\}.
    \]
    When $|Z(p)|=\ell$, we always write
    \[
     Z(p) = \{z_1,z_2,\ldots,z_{\ell} \}
    \]
    where the indices $z_1 < z_2 < \cdots < z_{\ell}$ are listed in increasing order.
\end{definition}

\begin{lemma}
    \label{lem:zero_equal_rmin}
    Let $\pi \in \perm{n}$ with Lehmer code $p \in \lehmer{n}$. Then 
    \[
    Z(p) = \rmin(\pi).
    \]
\end{lemma}

\begin{proof}
    We have $p_i = 0$ if and only if there is no entry $j > i$ such that $\pi_j<\pi_i$. 
    This is exactly the condition that $\pi_i$ is a right-to-left minimum.
\end{proof}


For example, when $\pi=52718364$ and $p=(4,1,4,0,3,0,1,0)$,
the zero set of $p$ is
$
Z(p)=\{4,6,8\}.
$
Looking at the permutation, the entries at these positions are $\{1,3,4\}$ and these are 
precisely the entries that are smaller than
everything to their right, as specified by Lemma  \ref{lem:zero_equal_rmin}.

Finally, we make an important observation about the reverse-complement map, connecting the left-to-right maxima of $\revcompl{\pi}$ with the right-to-left minima of $\pi$.

\begin{definition}
    \label{def:rev}
    For any subset $I \subseteq [n]$, define 
    \[
    \Rev(I) = \{ n + 1 - i : i \in I \}.
    \]
\end{definition}

\begin{lemma}
    \label{lem:rmin_to_lmax}
    Let $\pi=(\pi_1,\ldots,\pi_n)\in\perm{n}$. Then
    \[
    \Rev(\rmin(\pi)) = \lmax(\revcompl{\pi}).
    \]
\end{lemma}

\begin{proof}
    The reverse-complement mapping $\pi \mapsto \revcompl{\pi}$ sends the point $(i,\pi_i)$ in the plot of $\pi$ to
        $  
        (n+1-i,n+1-\pi_i)
        $
    in the plot of $\revcompl{\pi}$.
    If $i\in\rmin(\pi)$, then $\pi_i<\pi_j$ for every $j>i$. After applying the reverse-complement to each $(j,\pi_j)$, every such point lies to the left of position $n+1-i$ and has smaller value than $n+1-\pi_i$. Hence $n+1-i\in\lmax(\revcompl{\pi})$.
    
    The same argument applied backward gives the reverse inclusion.
\end{proof}




For example, when 
$\pi=52718364$, we have $\rmin(\pi)=\{4,6,8\}$, and hence 
$\Rev(\rmin(\pi)) = \{1,3,5\}.
$
On the other hand,
$\revcompl{\pi}=53618274$, and its left-to-right maxima occur at positions
$
\lmax(\revcompl{\pi})=\{1,3,5\},
$
in accordance with Lemma \ref{lem:rmin_to_lmax}.

%% file: section-32-1.tex
\section{Lehmer Codes for \underline{32}1-Avoiding Permutations}

\label{sec:32-1-permutations}

We turn our attention to the family of $\underline{32}1$-avoiding permutations, whose
Lehmer-code characterization and maximality criterion were established in
\cite{BHR}. We reformulate that criterion into a
form that is better suited for our proof of Theorem \ref{thm:gravity_bijection}. We then show that each maximal Lehmer code in this family is uniquely determined by its zero set, and  translate this zero set description into an equivalent descent set description. Next, we show that $\maxset{n}{\underline{32}1}$ is enumerated by the Fibonacci numbers.  We close with a lemma that specifies the permutation values on the zero set of $p \in \maxset{n}{\underline{32}1}$. 

We start with the characterization of the Lehmer codes for $\underline{32}1$-avoiding permutations from $\cite{BHR}$. We include its short proof for completeness. See Figure
\ref{fig:32-1-3-21-posets} for a visualization of $\lehmerpat{4}{\underline{32}1}$.

\begin{lemma}[\cite{BHR}, Lemma 2.3]
    \label{lem:condition_32_1}
    The Lehmer code $p=(p_1,\ldots,p_n) \in \lehmerpat{n}{\underline{32}1}$ if and only if the following condition holds: 
    \begin{equation}
    \label{eqn:condition_32_1}
            \mbox{if } j \in D(p) \mbox{ then } p_{j+1}=0.
    \end{equation}
\end{lemma}

\begin{proof}
For $1 \leq i \leq n-1$, observe that $\pi_j > \pi_{j+1}$ if and only if $p_j > p_{j+1}$.
Suppose that $p$ follows condition \eqref{eqn:condition_32_1}.  Suppose $\pi_j > \pi_{j+1}$, so that  $j \in D(p)$ and therefore $p_{j+1}=0$. So $\pi_{j+1} < \pi_k$ for $j+1 < k \leq n$, which means that $j$ cannot be the start of a $\underline{32}1$-pattern.
Now consider $\pi \in \permpat{n}{\underline{32}1}$, and suppose that $j \in D(p)$, so that $\pi_{j} > \pi_{j+1}$ as well. To avoid creating a $\underline{32}1$-pattern, we must have $\pi_{j+1} > \pi_k$ for $j+1 < k \leq n$, which means that $p_{j+1}=0$.
\end{proof}

\begin{cor}
\label{cor:32-1-descent-increase-by-2}
If $D(p) = \{ d_1, \ldots, d_m \}$ is the descent set of $p \in \lehmerpat{n}{\underline{32}1}$, then $d_{i+1}-d_i \geq 2$.
\end{cor}

\begin{proof}
If $p$ has a descent at index $j$, then $p_{j+1}=0$, so index $j+1$ cannot be a descent.
\end{proof}

The maximal Lehmer codes have the following form. This criterion is different from the one found in Lemma 2.7 of \cite{BHR}, but it is straightforward to show that they are equivalent. Figure \ref{fig:32-1-example} shows the Lehmer codes of $\maxset{5}{\underline{32}1}$ along with their corresponding permutations, descents sets and zero sets.

\begin{figure}[ht]
\begin{center}
\begin{tikzpicture}[scale=.6]

\begin{scope}

    \node at (-3,2.5) {\footnotesize $\pi$};
    \node at (-3,0.5) {\footnotesize $p$};
    \node at (-3,-1) {\footnotesize $D(p)$};
    \node at (-3,-2.5) {\footnotesize $Z(p)$};

        \node at (0,2.5) {\footnotesize $51423$};
        \newtri[(0, 0)]{4,0,2,0}{orange};
        \node at (0,-1) {\footnotesize $\{1,3\}$};
        \node at (0,-2.5) {\footnotesize $\{2,4,5\}$};
        
        \node at (3,2.5) {\footnotesize $51342$};
        \newtri[(3, 0)]{4,0,1,1}{orange};
        \node at (3,-1) {\footnotesize $\{1,4\}$};
        \node at (3,-2.5) {\footnotesize $\{2,5\}$};        

        \node at (6,2.5) {\footnotesize $45132$};
        \newtri[(6, 0)]{3,3,0,1}{orange};
        \node at (6,-1) {\footnotesize $\{2,4\}$};
        \node at (6,-2.5) {\footnotesize $\{3,5\}$};        

        \node at (9,2.5) {\footnotesize $34512$};        
        \newtri[(9, 0)]{2,2,2,0}{orange};
        \node at (9,-1) {\footnotesize $\{3\}$};
        \node at (9,-2.5) {\footnotesize $\{4,5\}$};
        
        \node at (12,2.5) {\footnotesize $23451$};        
        \newtri[(12, 0)]{1,1,1,1}{orange};
        \node at (12,-1) {\footnotesize $\{4\}$};
        \node at (12,-2.5) {\footnotesize $\{5\}$};        

\end{scope}

\end{tikzpicture}    

\caption{The maximal Lehmer codes $\maxset{5}{\underline{32}1}$. The triangle representation of each code $p$ is accompanied by its corresponding permutation $\pi$,  descent set $D(p)$ and zero set $Z(p)$.}

\label{fig:32-1-example}

\end{center}    
\end{figure}


\begin{lemma}
\label{lem:maximal_32_1_lehmer_condition}
The Lehmer code $p \in \maxset{n}{\underline{32}1}$ if and only if the following conditions hold for $p$ and its descent set $D(p) = \{ d_1, \ldots, d_m \}$.
\begin{enumerate}[(a)]
    \item The last descent $d_m \in \{ n-2, n-1 \}$.
    \item For $1\le j\le m$, we have $p_{d_j} = n-d_j$ and $p_{d_j+1} = 0$. 
    \item The code $p$ is constant (and nonzero) on the subintervals
    \[
    [1,d_1], \quad [d_1+2,d_2], \quad [d_2+2,d_3], \quad  \ldots, \quad [d_{m-1}+2, d_m]
    \]
    and zero elsewhere.
    
\end{enumerate}
\end{lemma}

\begin{proof}
By Lemma~\ref{lem:condition_32_1}, a code $p \in \lehmerpat{n}{\underline{32}1}$ if and only if every descent is followed by a zero. 
Suppose first that $p\in\maxset{n}{\underline{32}1}$. 

If $d\in D(p)$, then $p_{d+1}=0$ by Lemma~\ref{lem:condition_32_1}. Maximality forces $p_d=n-d$: otherwise we could increase $p_d$ and $d$ would still be a descent followed by zero. This proves (b).

Next, we observe that since $p_n=0$, maximality forces a descent at either $n-1$ or $n-2$. Indeed, if $p_{n-1}>0$, then $n-1\in D(p)$. If $p_{n-1}=0$, then maximality forces $p_{n-2}=2$, so $n-2\in D(p)$. Hence $d_m\in\{n-2,n-1\}$, proving (a).

Finally, by Corollary \ref{cor:32-1-descent-increase-by-2}, the descents of $p$ are not consecutive. We claim that the only location where we can have consecutive zeros is indices $n-1$ and $n$.
Indeed, if $p_i = p_{i+1} = p_{i+2} = 0$ then we can increase $p_{i+1}$ to obtain another code in $\lehmerpat{n}{\underline{32}1}$, a contradiction. Otherwise, if $p_i=p_{i+1}=0$ and $p_{i+1}>0$ then we can increase $p_{i}$ to equal $p_{i+1}$ to obtain another code in $\lehmerpat{n}{\underline{32}1}$, a contradiction.
Therefore the positive entries occur in blocks
\[
[1,d_1],\ [d_1+2,d_2],\ldots,\ [d_{m-1}+2,d_m],
\]
and each such block is weakly increasing. Maximality forces each block to be constant: if $p_i<p_{i+1}$ inside a block, then increasing $p_i$ to $p_{i+1}$ would create no new descent and would keep the code $\underline{32}1$-avoiding, as per Lemma \ref{lem:condition_32_1}. Thus (c) holds.

Conversely, suppose that (a), (b), and (c) hold for code $p \in \lehmer{n}$. Then every descent of $p$ is one of the indices $d_j$, and it is followed by a zero. Hence $p\in\lehmerpat{n}{\underline{32}1}$ by Lemma~\ref{lem:condition_32_1}.

It remains to show maximality. Let $s\in\lehmerpat{n}{\underline{32}1}$ with $p\preceq s$. 
Suppose that $d \in D(s)$ and $s_{d+1} > 0 = p_{d+1}$. Then we still have $s_{d} = p_{d} = n-d > s_{d+1}$, so
$d \in D(s)$ but $s_{d+1} \neq 0$, contradicting equation \eqref{eqn:condition_32_1}.
Thus $s_{d+1} = 0$ for all $d \in D(p)$.
On each positive block of $p$, the right endpoint value is already $p_{d_j} = n-d_j$, the largest possible value at position $d_j$, and the block is constant; increasing any earlier entry in the block would create a descent not followed by zero. Therefore no coordinate of $p$ can be increased while remaining in $\lehmerpat{n}{\underline{32}1}$. Hence $p=s$, so $p\in\maxset{n}{\underline{32}1}$.
\end{proof}

\begin{cor}
\label{cor:zero-set-32-1-avoiding}
Let $p\in \maxset{n}{\underline{32}1}$ with descent set
$D(p) = \{ d_1, d_2, \ldots, d_m \}$. Then
$$
Z(p) =
\begin{cases}
    D'(p) & \mbox{if } d_m = n-1, \\
    D'(p) \cup \{ n \} &  \mbox{if } d_m = n-2.
\end{cases}
$$
where $D'(p) = \{ d_i + 1 : 1 \leq i \leq m \}$.
\end{cor}

\begin{proof}
By condition (b), we have $D'(p) \subseteq Z(p)$. 
In addition, we always have $p_n =0$; this index is included in $D'(p)$ when $d_m=n-1$.  Otherwise, we have $d_m=n-2$ and we must explicitly add index $n$ to our set of zeros.
\end{proof}

\begin{cor}
\label{cor:zero_set_determines_maximal_32-1}
A code $p \in \maxset{n}{\underline{32}1}$ is uniquely determined by its zero set $Z(p)$.
\end{cor}

\begin{proof}
Consider $p \in \maxset{n}{\underline{32}1}$ with zero set
$Z(p) = \{ z_1, z_2, \ldots z_{\ell} \}$.
Note that $z_{\ell}=n$, so $Z(p) \neq \emptyset$. Set $z_0=0$ for convenience. By Lemma \ref{lem:maximal_32_1_lehmer_condition}, we have $p_k = n-d_i = n - z_i -1$ for $z_{i-1} < k < z_i$. This completely determines the remaining entries of $p$.
\end{proof}

We spend the remainder of this section discussing the descent set $D(p)$ and the zero set $Z(p)$ for $p \in \maxset{n}{\underline{32}1}$.
Recall from Definition \ref{def:descent-set-map} that
$\descpat{n}{\underline{32}1} = \{ D(p) : p \in \maxset{n}{\underline{32}1} \}$ is the collection of all descent sets for Lehmer codes in $\maxset{n}{\underline{32}1}$, and that     
the mapping
$
\descmap{\underline{32}1}: \maxset{n}{\underline{32}1} \rightarrow \descpat{n}{\underline{32}1}
$
is given by $\descmap{\underline{32}1}(p) = D(p)$.

\begin{lemma}
\label{lem:delta_32-1_bijection}
    The mapping
    $\descmap{\underline{32}1}: \maxset{n}{\underline{32}1} \rightarrow \descpat{n}{\underline{32}1}$ is a bijection
\end{lemma}

\begin{proof}
The map is surjective by the definition of $\descpat{n}{\underline{32}1}$.
It remains to prove injectivity.
Let $p,s\in\maxset{n}{\underline{32}1}$ and suppose $D(p)=D(s)$.
By Corollary \ref{cor:zero-set-32-1-avoiding}, we have $Z(p) = Z(s)$ and then by Corollary \ref{cor:zero_set_determines_maximal_32-1}, we have $p=s$. 
\end{proof}


\begin{lemma}
\label{lem:descent_sets_32_1_characterization}
    For $n \ge 2$,
    \[
    \descpat{n}{\underline{32}1} = \{ \{d_1, \cdots, d_m\}\subseteq [n-1]:
     d_m\in\{n-2,n-1\} \mbox{ and } d_{j+1}-d_j\ge 2 \}.
    \]
\end{lemma}
\begin{proof}
    Let $p \in \maxset{n}{\underline{32}1}$, and let $D(p) = \{d_1 , \ldots , d_m\}$. 
    By Lemma \ref{lem:maximal_32_1_lehmer_condition}(a) we have $d_m \in \{ n-2, n-1 \}$.
    By Corollary \ref{cor:32-1-descent-increase-by-2}, $D(p)$ has no consecutive elements. 
    


    Conversely, let $D=\{d_1, \cdots, d_m\}\subseteq[n-1]$ 
    be a set with no consecutive elements and such that $d_m \in \{n-2,n-1\}$.     
    Define $p=(p_1,\ldots,p_n)$ by the formula
   \[
    p_i=
    \begin{cases}
    n-d_j & \mbox{if } d_{j-1}+2 \le i\le d_j \text{ for some } 1\le j\le m,\\
    0 & \mbox{otherwise},
    \end{cases}
    \]
    where we define $d_0 = -1$ for convenience.
    If $p$ is a valid Lehmer code, then it satisfies the conditions of Lemma \ref{lem:maximal_32_1_lehmer_condition} by construction. So we need only show that $0 \leq p_i \leq n-i$ for $1 \leq i \leq n$.
    If $p_i > 0$, then $p_i = n -d_j$ for some $j$ with $i \le d_j$, so $p_i=n-d_j\le n-i.$ If $p_i = 0$, then the Lehmer bound is automatic. 
%
%
%
\end{proof}

The following result is Lemma 2.9 in \cite{BHR}. We give an alternate proof using descent sets that $\maxset{n}{\underline{32}1}$ is enumerated by the Fibonacci number $F_n$, where $F_1=1$, $F_2=1$ and $F_{n} = F_{n-1} + F_{n-2}$ for $n \geq 3$.

\begin{lemma}
    \label{lem:maximal_32_1_number}
    We have $|\maxset{n}{\underline{32}1}| = F_{n}$, the $n$th Fibonacci number.
\end{lemma}

\begin{proof}
By Lemma \ref{lem:delta_32-1_bijection}, we have
$|\maxset{n}{\underline{32}1}| = |\descpat{n}{\underline{32}1}|$. 

Observe that $\descpat{2}{\underline{32}1} = \{ \{1\} \} $ and $\descpat{3}{\underline{32}1} = \{ \{1 \}, \{2 \} \}$. So $|\descpat{2}{\underline{32}1}| = 1 = F_2$ and 
$|\descpat{3}{\underline{32}1}| = 2 = F_3$.
For $n > 3$ the collection of sets in $\descpat{n}{\underline{32}1}$ containing element $1$ are in bijection with $\descpat{n-2}{\underline{32}1}$, and the collection of sets in $\descpat{n}{\underline{32}1}$ that do not contain element $1$ are in bijection with $\descpat{n-1}{\underline{32}1}$. Hence
$$
|\descpat{n}{\underline{32}1}| = 
|\descpat{n-1}{\underline{32}1}| + |\descpat{n-2}{\underline{32}1}|
\quad \mbox{for} \quad n \geq 3,
$$
which is the Fibonacci recurrence.
\end{proof}

We conclude this section with an important lemma about the zero set for a maximal $\underline{32}1$-avoiding permutation. 

\begin{lemma}
\label{lem:zero_set_entry_incresaing_by_1}
    Let $p\in \maxset{n}{\underline{32}1}$ and let $\pi \in \permpat{n}{\underline{32}1}$ be the permutation such that $p = L(\pi)$. If
    \[
    Z(p)= \{i \in [n]: p_i = 0\} = \{z_1,z_2,\ldots,z_{\ell}\},
    \]
    then $\pi_{z_k}=k$, for $ 1\le k\le \ell$.
\end{lemma}

For example, consider $p = (6,6,6,0,3,3,0,1,0) \in \maxset{9}{\underline{32}1}$ with corresponding permutation $\pi=789156243$. Its zero set is $Z(p)=\{4,7,9\}$, and we have $\pi_4=1$, $\pi_7=2$ and $\pi_9=3$.

\begin{proof}
    Set $z_0=0$. By Lemma \ref{lem:maximal_32_1_lehmer_condition}, for every
    $1\le k\le \ell$ and $i \in [z_{k-1}+1, z_k-1]$, we have $p_i=n-z_k+1.$
    Recall from Section \ref{sec:decode} that when a permutation is reconstructed from its Lehmer code, the entry $\pi_i$ is chosen to be the $(p_i+1)$th smallest unused value.
    
    We prove by induction on $k$ that $\pi_{z_k}=k$. More precisely, we prove a stronger statement: before position $z_k$ is decoded, the values $1,\ldots,k-1$ have already been used, while $k,\ldots,\ell$ are still unused.
    
    For $k=1$, we have $p_i=n-z_1+1$ for $i \in [1, z_1 -1]$. Since $z_1,\ldots,z_{\ell}$ are $\ell$ positions in $[z_1,n]$, we have $\ell \le n-z_1+1 = p_i$. Hence every entry before $z_1$ skips the first $p_i \geq \ell$ unused values, so none of $1,\ldots,\ell$ is used before $z_1$. Since $p_{z_1}=0$, we get $\pi_{z_1}=1$.
        
    Now let $k > 1$, and assume that $\pi_{z_j} = j$ for all $1 \leq j \leq k-1$ and that  $\pi_i \notin \{k, \ldots, \ell\}$ for $i \in [1, z_{k-1}]$. 
    We show that no entry of $\pi$ between $z_{k-1}$ and $z_k$ uses any of the values $k,\ldots,\ell$. For every $i \in [z_{k-1}+1, z_k-1]$ we have $p_i = n-z_k + 1$. Because $z_k, \dots, z_{\ell}$ are $\ell-k + 1$ positions in $[z_k,n]$, we have
    $
    \ell-k+1 \le n- z_k + 1 = p_i.
    $
    Thus $\pi_i$ is chosen as the $(p_i+1)$st smallest unused value, so it skips at least the first $p_i \geq \ell-k+1$ unused values, including $k,\ldots,\ell$. Hence the values $k,\ldots,\ell$ remain unused until position $z_k$. Since $p_{z_k}=0$, the entry $\pi_{z_k}$ is the smallest unused value, namely $k$.
    Therefore, by induction, $\pi_{z_k}=k$ for all $1\le k\le \ell$.
\end{proof}

%% file: section-3-21.tex
\section{Lehmer Codes for 3\underline{21}-Avoiding Permutations}

\label{sec:3-21-permutations}

This section follows a structure that is parallel to the previous section.
We provide characterizations of the Lehmer codes in $\lehmerpat{n}{3\underline{21}}$ and the maximal Lehmer codes in $\maxset{n}{3\underline{21}}$. We then show that a maximal Lehmer code $p \in \maxset{n}{3\underline{21}}$ is completely determined by its descent set $D(p)$. 
Next, we show that the mapping $\F$ from equation \eqref{eqn:gravity_bijection} is a bijection, and hence that $\maxset{n}{3\underline{21}}$ is enumerated by the Fibonacci numbers.  We close with a lemma that specifies the permutation values on the descent set of $p \in \maxset{n}{3\underline{21}}$.


We start by giving the Lehmer code condition for $\lehmerpat{n}{3\underline{21}}$-avoiding permutation. See Figure
\ref{fig:32-1-3-21-posets} for a visualization of $\lehmerpat{4}{3\underline{21}}$.

\begin{lemma}
\label{lem:condition_3_21}
    The Lehmer code $p=(p_1, \ldots, p_n) \in \lehmerpat{n}{3\underline{21}}$  if and only if the following condition holds: 
    \begin{equation}
    \label{eqn:condition_3_21}
    \mbox{if } j\in D(p) \mbox{ then }
    p_j+j>p_i+i \mbox{ for } 1\le i<j.
    \end{equation}
\end{lemma}


\begin{proof}
Let $p=(p_1, \ldots, p_n)$ be the Lehmer code of $\pi = \pi_1 \cdots \pi_n$.
We prove the contrapositive: $\pi$ has a $3\underline{21}$ pattern if and only if there exists $1 \leq i < j$ such that $p_j > p_{j+1}$ and $p_j+j \leq p_i+i$.

 Assume  there exists $\pi_i>\pi_j>\pi_{j+1}$ with $1\le i <j$. Observe that  $p_j > p_{j+1}$. Without loss of generality, $\pi_i = \max \{ \pi_1, \ldots, \pi_{j-1}\}$, so that $p_k < p_i$ for $i < k \leq j$. Consequently, $p_i \ge p_j + (j-i)$, or in other words, $p_i+i \ge p_j + j$.

Next, assume there exists $p_i$ and $p_j > p_{j+1}$ with $1\le i<j$ and $p_j+j\le p_i+i$, equivalently $p_j \le p_i - (j-i)$. In between $i$ and $j$, we have accounted for at most $j-i-1$ available elements below $\pi_i$. So  $\pi_j < \pi_i$ and clearly $\pi_j > \pi_{j+1}$, which creates a $3\underline{21}$ pattern at $i$, $j$, $j+1$. 
\end{proof}

\begin{cor}
\label{cor:descent-increase-by-2}
If $D(p) = \{ d_1, \ldots, d_m \}$ is the descent set of $p \in \lehmerpat{n}{3\underline{21}}$, then $d_{i+1}-d_i \geq 2$.
\end{cor}

\begin{proof}
Suppose that $p$ has descents at $i$ and $i+1$. Then
$
p_i+i \geq (p_{i+1} + 1) + i = p_{i+1} + (i+1),
$
which contradicts condition \eqref{eqn:condition_3_21} for $j=i+1$.
\end{proof}


Next, we characterize the Lehmer codes for $\maxset{n}{3\underline{21}}$. Figure \ref{fig:3-21-example} shows the Lehmer codes of $\maxset{5}{\underline{32}1}$ along with their corresponding permutations and descents sets.

\begin{figure}[ht]
\begin{center}
\begin{tikzpicture}[scale=.6]

\begin{scope}

    \node at (-3,2.5) {\footnotesize $\pi$};
    \node at (-3,0.5) {\footnotesize $p$};
    \node at (-3,-1) {\footnotesize $D(p)$};

        \node at (0,2.5) {\footnotesize $34251$};
        \newtri[(0, 0)]{2,2,1,1}{teal};
        \node at (0,-1) {\footnotesize $\{2,4\}$};
        
        \node at (3,2.5) {\footnotesize $42351$};
        \newtri[(3, 0)]{3,1,1,1}{teal};
        \node at (3,-1) {\footnotesize $\{1,4\}$};

        \node at (6,2.5) {\footnotesize $43512$};
        \newtri[(6, 0)]{3,2,2,0}{teal};
        \node at (6,-1) {\footnotesize $\{1,3\}$};

        \node at (9,2.5) {\footnotesize $45123$};        
        \newtri[(9, 0)]{3,3,0,0}{teal};
        \node at (9,-1) {\footnotesize $\{2\}$};
        
        \node at (12,2.5) {\footnotesize $51234$};        
        \newtri[(12, 0)]{4,0,0,0}{teal};
        \node at (12,-1) {\footnotesize $\{1\}$};

\end{scope}

\end{tikzpicture}    

\caption{The maximal Lehmer codes $\maxset{5}{3\underline{21}}$. The triangle representation of each code $p$ is accompanied by its corresponding permutation $\pi$ and descent set $D(p)$.
}

\label{fig:3-21-example}

\end{center}    
\end{figure}

\begin{lemma}
\label{lem:maximal_condition_3_21}
The Lehmer code $p=(p_1, \ldots, p_n) \in \maxset{n}{3\underline{21}}$ if and only if the following conditions hold for $p$ and its descent set $D(p) = \{ d_1, \ldots, d_m \}.$
\begin{enumerate}
    \item[(a)] The first descent $d_1 \in \{1,2\}$.

    \item[(b)] We have $p_{d_j} + d_j =n-(m-j)$ for $1\le j\le m$.

    \item[(c)] The code $p$ is constant on the subintervals 
            $$[1, d_1], \quad  [d_1+1, d_2], \quad [d_2+1, d_3], \quad  \ldots, \quad  [d_{m-1}+1, d_m],  \quad [d_m+1, n].$$
\end{enumerate}
\end{lemma}


\begin{proof}
    Let $p=(p_1, \ldots, p_n) \in \lehmer{n}$ be 3\underline{21}-avoiding with descent set $D(p) = \{d_1, \cdots, d_m\}$. 
    
    
    We prove the forward direction by contrapositive. Suppose that at least one of (a), (b), and (c) fails. We show that $p$ is not maximal.

    First, suppose that (a) is false, so that $d_1 > 2$. Define $p'= (p'_1,\ldots,p'_n)$ by
    \[
    p'_k= \begin{cases} p_1+1, & k=1,\\
    p_k, & \text{otherwise}. \end{cases}
    \]
   Observe that $p_1 \leq p_{d_1}$ since there are no descents in $[1,d_1-1]$. We have
    $$
    p'_1+ 1 = p_1+2 \leq p_{d_1}+2 < p_{d_1}+d_1 = p'_{d_1} + d_1
    $$
    and all of the remaining conditions \eqref{eqn:condition_3_21} hold as well.
    We have $p \prec p'$,  hence $p$ is not maximal.
    
    Second, suppose that (b) fails. Assume that $p_{d_{\ell}} + d_{\ell} \neq n-(m-{\ell})$, for some $d_{\ell} \in D(p)$. There are two cases.
    
    \textit{Case 1.} Suppose that $p_{d_{\ell}} + d_{\ell} < n-(m-\ell)$. Without loss of generality, choose $\ell$ so that  $p_{d_a} + d_a= n - (m - a)$, for $\ell < a \leq m$. Define $p'=(p'_1,\ldots,p'_n)$ by
    \[
    p'_k=
    \begin{cases}
    p_k+ 1 & \text{if } k \in [d_{\ell-1}+1, d_{\ell}] \\
    p_k & \text{otherwise}.
    \end{cases}
    \]
    For $p'_{d_b}$, with $b < {\ell}$,
    \[
    p'_{d_b} + d_b = p_{d_b} + {d_b} < p_{d_{\ell}} + d_{\ell} < p_{d_{\ell}} +1 + d_{\ell} = p'_{d_{\ell}} + d_{\ell}.
    \]
    For $p'_{d_b}$, with $b \ge {\ell} + 1$,
    \[
    p'_{d_{\ell}} + d_{\ell} \le n-(m-{\ell}) < n-(m-{\ell} - 1) = p_{d_{{\ell}+1}} + d_{{\ell}+1} \le p_{d_b} + d_b  = p'_{d_b} + d_b.
    \]
    Since all entries of $p'$ are the same as $p$ except for $d_{{\ell}-1} < k \le d_{{\ell}}$, the above two inequalities show $p'$ is also $3\underline{21}$-avoiding. Thus,  $p \prec p'$, so $p$ is not maximal.
    
    \textit{Case 2.} Suppose that $p_{d_{\ell}}+ d_{\ell} > n-m+{\ell}$. By condition \eqref{eqn:condition_3_21}, the numbers $p_{d_{\ell}}+d_{\ell},p_{d_{{\ell}+1}}+d_{{\ell}+1},\ldots,p_{d_m}+d_m$
    are $m-{\ell}+1$ distinct integers all at most $n$. These $m - {\ell} + 1$ distinct entries must all lie in the set $\{n-m+{\ell}+1,\ldots,n\}$ which has only $m-{\ell}$ elements. This is impossible by pigeonhole principle.

    Third, suppose that (c) is false. Then there exists an index $i\notin D(p)$ such that $p_i\neq p_{i+1}$. Since $i\notin D(p)$, we have $p_i<p_{i+1}$. Define $p'=(p'_1,\ldots,p'_n)$ by
    \[
    p'_k= \begin{cases} p_{i+1} & \text{if } k=i,\\
    p_k & \text{otherwise}. \end{cases}
    \]
    Then $p'$ is still a Lehmer code, since
    $
    p'_i=p_{i+1}\le n-(i+1)<n-i.
    $
    Moreover, $p\prec p'$.
    It remains to check that $p'$ still satisfies condition \eqref{eqn:condition_3_21}.  The only change is the increase of the $i$-th entry. Since
    $
    p'_i=p_{i+1}=p'_{i+1},
    $
    no descent is created at position $i$. Also, no new descent is created at position $i-1$. Thus the only possible issues are for $j \in D(p')$ with $j > i$. 
 Since $p'$ agrees with $p$ at positions $j$ and $j+1$, we have $j \in D(p)$. Hence, using condition \eqref{eqn:condition_3_21} for $p$, we have
    \[
    p'_i+i = p_{i+1}+i < p_{i+1}+i+1 \le p_j+j= p'_j+j.
    \]
    All other inequalities in \eqref{eqn:condition_3_21} are unchanged. Therefore
    $
    p'\in \lehmerpat{n}{3\underline{21}},
    $
    which contradicts the maximality of $p$. Hence condition (c) must hold.
This completes the proof of the forward direction.


    
    Next, suppose that conditions $(a)$, $(b)$, and $(c)$ hold for $p$. We prove that $p\in \maxset{n}{3\underline{21}}$. Assume for the sake of contradiction that there exists $s=(s_1,\ldots,s_n)\in \lehmerpat{n}{3\underline{21}}$ such that $p\prec s$. Let $\ell=\max\{i:s_i>p_i\}$, so that $s_i=p_i$ for all $i>\ell$. There are four cases.
    
    First, suppose that $\ell=d_k$ for some $1\le k\le m$. Since $p_{d_m}=n-d_m$ by condition $(b)$, we must have $k<m$. Note that $s_{d_k}>p_{d_k}$ and all entries after index $d_k$ are unchanged, so we have
    \[
    s_{d_k}+d_k \ge p_{d_k}+d_k+1 = n-m+k+1 = p_{d_{k+1}}+d_{k+1} = s_{d_{k+1}}+d_{k+1},
    \]
    which violates condition \eqref{eqn:condition_3_21} for $s$ with $i=d_k$ and $j=d_{k+1}$.
    
    Second, suppose that $\ell>d_m$. Since $p_{d_m}=n-d_m$, we have $s_{d_m}=p_{d_m}$. By condition $(c)$, $p$ is constant on $[d_m+1,n]$. Since $p_n=0$, this gives $p_i=0$ for all $i>d_m$. Hence
    \[
    s_\ell>s_{\ell+1}=0,
    \]
    so $\ell$ is a descent of $s$. But
    \[
    s_\ell+\ell\le n=s_{d_m}+d_m,
    \]
    which violates condition \eqref{eqn:condition_3_21} for $s$  with $i=d_m$ and $j=\ell$.
    
    Third, suppose that $d_{j-1}<\ell<d_j$ for some $j\ge 2$. Since $p$ is constant on $[d_{j-1}+1,d_j]$ by condition $(c)$, we have $p_\ell=p_{\ell+1}$. Also $s_{\ell+1}=p_{\ell+1}$ because $\ell$ is the largest index where $p$ and $s$ differ. Therefore
    $
    s_\ell>p_\ell=p_{\ell+1}=s_{\ell+1},
    $
    so $\ell$ is a descent of $s$. Condition \eqref{eqn:condition_3_21} gives
    \[
    s_\ell+\ell>s_{d_{j-1}}+d_{j-1} \ge p_{d_{j-1}}+d_{j-1} = n-m+j-1.
    \]
    Thus $s_\ell+\ell\ge n-m+j$. Since $d_j>\ell$, the entry at $d_j$ is unchanged, so
    \[
    s_{d_j}+d_j = p_{d_j}+d_j = n-m+j \le s_\ell+\ell,
    \]
    which violates condition \eqref{eqn:condition_3_21} for $s$  with $i=\ell$ and $j=d_j$.
    
    Finally, suppose that $\ell<d_1$. By condition $(a)$, the only possible case is $\ell=1<d_1=2$. Since $p$ is constant on $[1,d_1]$, we have $p_1=p_2$. Thus
    $
    s_1>p_1=p_2=s_2,
    $
    so $1 \in D(s)$. The descent at $d_2$ is unchanged, so $2 \in D(s)$ as well, which contradicts Corollary~\ref{cor:descent-increase-by-2}. Therefore no such $s$ exists, and $p$ is maximal.
\end{proof}




Lemma~\ref{lem:maximal_condition_3_21} shows that a maximal Lehmer code $p \in \maxset{n}{3\underline{21}}$ is completely determined by its descent set. 
Analogous to the previous section, we formulate this observation using the mapping
$
\descmap{3\underline{21}}: \maxset{n}{3\underline{21}} \rightarrow \descpat{n}{3\underline{21}}
$
given by $\descmap{3\underline{21}}(p) = D(p)$.

\begin{lemma}
\label{lem:delta_3-21_bijection}
    The mapping
    $\descmap{3\underline{21}}: \maxset{n}{3\underline{21}} \rightarrow \descpat{n}{3\underline{21}}$ is a bijection.
\end{lemma}

\begin{proof}
    The map is surjective by the definition of $\descpat{n}{3\underline{21}}$. It remains to show injectivity.

    Let $p, s \in \maxset{n}{3\underline{21}}$ and suppose $D(p) = D(s) = \{d_1, \cdots, d_m\}$. By Lemma \ref{lem:maximal_condition_3_21}, both $p$ and $s$  are constant on the disjoint subintervals
    $$
    [1, d_1],\quad [d_1+1, d_2],\quad [d_2+1, d_3],\quad \ldots \quad , \quad [d_{m-1}+1, d_m], \quad [d_m+1, n].
    $$
    Furthermore, $p_{d_j}=s_{d_j}=n-m+j-d_j$ for every $1\le j\le m,$ and clearly $p_n=s_n=0$.
    That is, the endpoint values are equal for each constant subinterval, so $p=s$.
\end{proof}

We can now describe the collection of descent sets for these maximal codes. In particular, either 1 or 2 must be a descent, and descents cannot be adjacent.

\begin{lemma}
\label{lem:descent_sets_3_21_characterization}
We have
\[
\descpat{n}{3\underline{21}} = \{ \{d_1, \cdots, d_m\}\subseteq[n-1]:  d_1\in\{1,2\} \mbox{ and }  d_{j+1}-d_j\ge 2 \mbox{ for } 1 \leq j < m \}.
\]
\end{lemma}
\begin{proof}
    Let $p=(p_1,\ldots,p_n)\in \maxset{n}{3\underline{21}}$, and let $D(p)=\{d_1, \cdots, d_m\}.$
    By Lemma \ref{lem:maximal_condition_3_21} (a), we have $d_1\in\{1,2\}$. 
    By Corollary \ref{cor:descent-increase-by-2}, $D(p)$ has no consecutive elements.
    
    
    Conversely, let $D=\{d_1, \cdots, d_m\}$ be a set with no consecutive elements and such that $d_1 \in \{1,2\}$.     
    Define $p=(p_1,\ldots,p_n)$ by the formula
    $$
 p_i = 
 \begin{cases}
     p_{d_1} = n - d_1 - (m-1) & \text{if } i \in [1, d_1], \\
     p_{d_j} = n - d_j - (m-j) & \text{if } i \in [d_{j-1}+1, d_j] \text{ where } 1 < j < m, \\
     0 & \text{if } i \in [d_m+1, n].
 \end{cases}
 $$
    If $p$ is a valid Lehmer code, then by construction, $p$ satisfies the conditions of Lemma \ref{lem:maximal_condition_3_21}, so $p \in \maxset{n}{3\underline{21}}$ and $D(p)=D$. So we need only show that $0 \leq p_i \leq n-i$ for $1 \leq i \leq n$.
 %
    This clearly holds for $i \in [d_m+1, n]$. For convenience, define $d_0=0$ and consider index $i \in [d_{j-1}+1, d_j]$. Then
    \[
    p_i=n-d_j-(m-j)\le n-d_j\le n-i.
    \]
    Also, since $D$ has no consecutive elements, we have $d_j+2(m-j)\le d_m\le n-1,$ so that 
    $d_j \leq n - 1 - 2(m-j)$.
    Therefore
    \[
    p_i=n-d_j-(m-j)\ge 1+(m-j)>0.
    \]
    So $p$ is a valid Lehmer code.
 %
\end{proof}





Our next objective is to show that $\maxset{n}{3\underline{21}}$ is a Fibonacci family. We could give give an inductive proof, but we will take this opportunity to make some progress on proving Theorem \ref{thm:gravity_bijection} at the same time. To start, we discuss the mapping $\flip: \descpat{n}{\underline{32}1} \rightarrow \descpat{n}{3\underline{21}}$ from equation \eqref{eqn:psi}, namely $\flip(D) = \{ n-d : d \in D \}.$

\begin{lemma}
\label{lem:descent_sets_reflect}
The mapping $\flip: \descpat{n}{\underline{32}1} \rightarrow \descpat{n}{3\underline{21}}$ is a bijection.
\end{lemma}
\begin{proof}
Let $D\in\descpat{n}{\underline{32}1}$. By Lemma \ref{lem:descent_sets_32_1_characterization},
$D$ is nonempty, has no consecutive elements, and satisfies $\max D\in\{n-2,n-1\}.$
Set $D'=\flip(D)$. The map $d\mapsto n-d$ preserves pairwise distances, so $D'$ has no consecutive elements. Finally, $\min D'=n-\max D\in\{1,2\}.$ Hence, by Lemma \ref{lem:descent_sets_3_21_characterization},
$
D'\in\descpat{n}{3\underline{21}},
$
and the mapping $\flip$ is well-defined. The mapping is clearly invertible, so it is a bijection. 
\end{proof}

Next, we turn our attention to the mapping $\F : \maxset{n}{\underline{32}1} \rightarrow \maxset{n}{3\underline{21}}$ from equation \eqref{eqn:gravity_bijection}, namely
$\F =\descmap{3\underline{21}}^{-1} \circ \flip \circ \descmap{\underline{32}1}$.
We show that this mapping is a bijection, deferring its relationship to the reverse-complement mapping $\revkomplempty$ until the next section. 

\begin{lemma}
\label{lem:gravity-bijection}
The mapping $\F : \maxset{n}{\underline{32}1} \rightarrow \maxset{n}{3\underline{21}}$ is a bijection.
\end{lemma}

\begin{proof}
By Lemmas \ref{lem:delta_32-1_bijection}, \ref{lem:delta_3-21_bijection} and \ref{lem:descent_sets_reflect}, this composite mapping is a bijection.    
\end{proof}

Our Fibonacci result for $\maxset{n}{3\underline{21}}$ follows directly.

\begin{lemma}
    \label{lem:maximal_3_21_number}
    We have $|\maxset{n}{3\underline{21}}| = F_{n}$, the $n$th Fibonacci number.
\end{lemma}

\begin{proof}
By Lemmas \ref{lem:gravity-bijection} and \ref{lem:maximal_32_1_number}, we have $|\maxset{n}{3\underline{21}}| = |\maxset{n}{\underline{32}1}| = F_{n}$.
\end{proof}

We conclude this section with an important lemma about the descent set for a maximal $3\underline{21}$-avoiding permutation.

\begin{lemma}
    \label{lem:decrease_to_lmax}
    Let $\pi \in \permpat{n}{3\underline{21}}$ with corresponding Lehmer code $p=(p_1,\ldots,p_n)\in \maxset{n}{3\underline{21}}$ with descent set $D(p) = \{ d_1, \ldots, d_m \}$.
    Then 
    \[
    \pi_{d_j} = n-(m-j)
    \]
    for $1 \leq j \leq m$. Furthermore, 
        \[
        D(p)=
        \begin{cases}
            \lmax(\pi) & \mbox{if } d_1 = 1, \\
            \lmax(\pi) \setminus \{1\} & \mbox{if } d_1 = 2. \\ 
        \end{cases}
        \]
\end{lemma}

For an example with $1 \in D(p)$, consider $p = (6,5,5,3,3,3,0,0,0) \in \maxset{9}{3\underline{21}}$ with corresponding permutation $\pi=768459123$. Its descent set is $Z(p)=\{1,3,6\}$, and we have $\pi_1=7$, $\pi_3=8$ and $\pi_6=9$, and $\lmax(\pi)= \{1,3,6\}$, in accordance with the lemma. For an example with $2 \in D(s)$, consider $s=(5,5,4,4,2,2,2,0,0)$ with corresponding permutation $\sigma=675834912$. Its descent set is $Z(s)=\{2,4,7\}$, and we have $\sigma_2=7$,  $\sigma_4=8$ and $\sigma_7=9$, and $\lmax(\sigma) = \{1,2,4,7\}$, again adhering to the lemma.

\begin{proof}
    We first show that 
    \begin{equation}
    \label{eqn:decrease_to_lmax_1}
            \pi_{d_1}<\pi_{d_2}<\cdots<\pi_{d_m}.
    \end{equation}
    Suppose, for contradiction, that there exist $i < j$ such that $\pi_{d_i}>\pi_{d_j}.$
    Since $d_j\in D(p)$, we have $p_{d_j}>p_{d_j+1}.$ This descent in the Lehmer code gives $\pi_{d_j}>\pi_{d_j+1}.$ Therefore
    $
    \pi_{d_i}>\pi_{d_j}>\pi_{d_j+1},
    $
    and the entries at positions $d_i<d_j<d_j+1$ form a $3\underline{21}$ pattern, a contradiction. 
    
    Next we show that $D(p)\subseteq \lmax(\pi)$. Let $D(p) = \{ d_1, \ldots, d_m\}$ and set $d_0=0$ for convenience. By Lemma~\ref{lem:maximal_condition_3_21}(c), the Lehmer code $p$ is constant on the interval $[d_{j-1}+1,d_j]$ and $d_1\in\{1,2\}$. Equal Lehmer code values in consecutive positions decode to increasing permutation entries, so the corresponding permutation entries satisfy
    \begin{equation}
    \label{eqn:decrease_to_lmax_2}   
        \pi_{d_{j-1}+1}<\pi_{d_{j-1}+2}<\cdots<\pi_{d_{j}}.
    \end{equation}
    Together, equations \eqref{eqn:decrease_to_lmax_1} and \eqref{eqn:decrease_to_lmax_2}
    imply that $\pi_{d_j} = \max_{1 \leq \ell \leq d_j} \pi_{\ell}$. 
    Hence, $d_j\in \lmax(\pi)$ and therefore
    \[
    D(p) \subseteq \lmax(\pi).
    \]
    Now we can apply Lemma~\ref{lem:lehmer_equation_lmax} and Lemma~\ref{lem:maximal_condition_3_21}(b) to each index in $d_j \in D(p)$ to conclude that
    $$
    \pi_{d_j} = p_{d_j} + d_j 
    = n - (m-j).
    $$
    So $\pi_{d_1}, \pi_{d_2}, \ldots, \pi_{d_m}$ are the $m$ largest elements of $[n]$, in increasing order.  

    Finally, we have $d_1 \in \{1,2\}$ by Lemma \ref{lem:maximal_condition_3_21}(a). If $d_1=1$ then $D(p) = \lmax(\pi)$. However, if $d_1=2$, then $1 \in \lmax(\pi)$, so that
    $\lmax(\pi) = D(p) \cup \{ 1 \}$.
\end{proof}

%% file: section-gravity-bijection.tex
\section{Proof of Theorem \ref{thm:gravity_bijection}}

\label{sec:gravity-bijection}

We now prove our main theorem, which we restate for convenience.

\begin{theorem*}[Theorem \ref{thm:gravity_bijection}]
Let $\F: \maxset{n}{\underline{32}1} \rightarrow \maxset{n}{3\underline{21}}$
be given by
$$
\F := \descmap{3\underline{21}}^{-1} \circ \flip \circ \descmap{\underline{32}1}
$$
where $\flip(D) = \{ n-d : d \in D\}$. Then the mapping $\F$ is a bijection. Furthermore, $\F$ corresponds to taking the reverse-complement of the associated permutation. That is, for every $p \in \maxset{n}{\underline{32}1}$, we have
$$
\F(p) = \revkompl{p}.
$$
As a consequence, $w(p)= w(\F(p))$.
\end{theorem*}






\begin{figure}[ht]
\begin{center}
\begin{tikzpicture}[scale=.6]

\begin{scope}

    \node at (-3,2.5) {\footnotesize $\pi$};
    \node at (-3,0.5) {\footnotesize $p$};
    \node at (-3,-1) {\footnotesize $D(p)$};

        \node at (0,2.5) {\footnotesize $561423$};
        \newtri[(0, 0)]{4,4,0,2,0}{orange};
        \node at (0,-1) {\footnotesize $\{2,4\}$};

        \node at (3,2.5) {\footnotesize $614523$};
        \newtri[(3, 0)]{5,0,2,2,0}{orange};
        \node at (3,-1) {\footnotesize $\{1,4\}$};

        \node at (6,2.5) {\footnotesize $345612$};
        \newtri[(6, 0)]{2,2,2,2,0}{orange};
        \node at (6,-1) {\footnotesize $\{4\}$};

        \node at (9,2.5) {\footnotesize $615243$};        
        \newtri[(9, 0)]{5,0,3,0,1}{orange};
        \node at (9,-1) {\footnotesize $\{1,3,5\}$};

        \node at (12,2.5) {\footnotesize $456132$};        
        \newtri[(12, 0)]{3,3,3,0,1}{orange};
        \node at (12,-1) {\footnotesize $\{3,5\}$};

        \node at (15,2.5) {\footnotesize $561342$};
        \newtri[(15, 0)]{4,4,0,1,1}{orange};
        \node at (15,-1) {\footnotesize $\{2,5\}$};

        \node at (18,2.5) {\footnotesize $613452$};        
        \newtri[(18, 0)]{5,0,1,1,1}{orange};
        \node at (18,-1) {\footnotesize $\{1,5\}$};

        \node at (21,2.5) {\footnotesize $234561$};        
        \newtri[(21, 0)]{1,1,1,1,1}{orange};
        \node at (21,-1) {\footnotesize $\{5\}$};

\end{scope}

\begin{scope}[shift={(0,-5)}]

    \node at (-3,0.5) {\footnotesize $\F(p)$};
    \node at (-3,2.5) {\footnotesize $\flip(D(p))$};
    \node at (-3,-1) {\footnotesize $\revcompl{\pi}$};
    
        \newtri[(0, 0)]{3,3,2,2,0}{teal};
        \node at (0,2.5) {\footnotesize $\{2,4\}$};
        \node at (0,-1) {\footnotesize $453612$};

        \newtri[(3, 0)]{3,3,1,1,1}{teal};
        \node at (3,2.5) {\footnotesize $\{2,5\}$};
        \node at (3,-1) {\footnotesize $452361$};
        
        \newtri[(6, 0)]{4,4,0,0,0}{teal};
        \node at (6,2.5) {\footnotesize $\{2\}$};        
        \node at (6,-1) {\footnotesize $561234$};
        
        \newtri[(9, 0)]{3,2,2,1,1}{teal};
        \node at (9,2.5) {\footnotesize $\{1,3,5\}$};
        \node at (9,-1) {\footnotesize $435261$};        
        
        \newtri[(12, 0)]{4,3,3,0,0}{teal};
        \node at (12,2.5) {\footnotesize $\{1,3\}$};
        \node at (12,-1) {\footnotesize $546123$};
        
        \newtri[(15, 0)]{4,2,2,2,0}{teal};
        \node at (15,2.5) {\footnotesize $\{1,4\}$};
        \node at (15,-1) {\footnotesize $534612$};

        \newtri[(18, 0)]{4,1,1,1,1}{teal};
        \node at (18,2.5) {\footnotesize $\{1,5\}$};
        \node at (18,-1) {\footnotesize $523461$};
        
        \newtri[(21, 0)]{5,0,0,0,0}{teal};
        \node at (21,2.5) {\footnotesize $\{1\}$};
        \node at (21,-1) {\footnotesize $612345$};
  
\end{scope}

\end{tikzpicture}    

\caption{The bijection $\F$. We start with maximal Lehmer code $p \in \maxset{6}{\underline{32}1}$, and its corresponding permutation $\pi$. We map $p$ to its descent set $D(p) \in \descpat{6}{\underline{32}1}$. Then we apply the mapping $\flip$ to obtain the descent set $\flip(D(p)) \in \descpat{6}{3\underline{21}}$. Next, we convert that descent set into the maximal Lehmer code $\F(p) \in \maxset{6}{3\underline{21}}$, which corresponds to permutation $\revcompl{\pi}$.}

\label{fig:full-bijection}

\end{center}    
\end{figure}

Figure \ref{fig:full-bijection} shows an example of the mapping $\F:\maxset{6}{\underline{32}1} \rightarrow \maxset{6}{3\underline{21}}$, along with the corresponding starting and ending permutations.
We remind the reader that there is an intuitive geometric interpretation of the $\F$ mapping as a ``flip and fall'' operation, see Figure \ref{fig:reflect-gravity}. Starting with the triangle representation of $p \in \maxset{n}{\underline{32}1}$, we flip along the diagonal $y=x$, then allow the floating boxes to fall down to obtain $\F(p) \in \maxset{n}{3\underline{21}}$.

We take a moment to explain why this ``flip and fall'' interpretation holds. Suppose that $p \in \maxset{n}{\underline{32}1}$ with $D(p)=\{d_1, \ldots, d_m\}$. Then the corresponding $s = \F(p) \in \maxset{n}{3\underline{21}}$ has descent set $\{n-d_m, \ldots, n-d_1\}$. Considering the triangle representations, the $j$th descent of $p$ has coordinates $(d_j, n-d_j)$ by Lemma \ref{lem:maximal_32_1_lehmer_condition}(b). Meanwhile, the $(m+1-j)$th descent of $s$ has coordinates $(n-d_j, d_j - (j-1))$ by Lemma \ref{lem:maximal_condition_3_21}(b). The ``flip'' is $(d_j, n - d_j) \mapsto (n - d_j, d_j)$, which reflects across the diagonal $y=x$. The ``fall'' is $(n - d_j, d_j) \mapsto (n - d_j, d_j - (j-1))$, which packs the boxes into contiguous columns.


Our proof strategy is as follows. Starting with $p \in \maxset{n}{\underline{32}1}$, we set $s=\F(p)$ and $t=\revkompl{p}$. We then show the descent set equality $D(s)=D(t)$. Since $\maxset{n}{3\underline{21}}$ is in bijection with its descent sets $\descpat{n}{3\underline{21}}$, this will establish the theorem.






\begin{proof}[Proof of Theorem \ref{thm:gravity_bijection}]
    The mapping $\F$ is a bijection by Lemma \ref{lem:gravity-bijection}. 
    Turning to the final statement of the theorem, it is straightforward to show that $\inv(\rev{\pi}) = {n+1 \choose 2} - \inv(\pi)$ and that $\inv(\compl{\pi}) = {n+1 \choose 2} - \inv(\pi)$. Hence $\inv(\revcompl{\pi}) = \inv(\pi)$, which is equivalent to the statement $w(\revkompl{p}) = w(p)$ where $p=L(\pi)$, so that $w(p) = \sum_{i=1}^n p_i = \inv(\pi)$.
    So all that remains is to show that $\F(p) = \revkompl{p}$ for all $p \in \maxset{n}{\underline{32}1}$.

    Let $p = (p_1, \dots, p_n) \in \maxset{n}{\underline{32}1}$ with descent set
    $D(p) = \{ d_1, d_2, \ldots, d_m\}$.
    Let $s = \F(p)$ with corresponding permutation $\sigma$  and let $t = \revkompl{p}$ with corresponding permutation $\tau = \revcompl{\pi}$. 
    By Lemma \ref{lem:delta_3-21_bijection},
    we have $s=t$ if and only if 
    $
    D(s)=D(t),
    $
    so we turn our attention to these descent sets.
    
    First, we claim that
    \[
    D(s) =
    \begin{cases}
        \Rev(Z(p)) & \mbox{if } p_{n-1} = 1 > 0 = p_n, \\
        \Rev(Z(p) \setminus \{n\}) & \mbox{if } p_{n-1} = 0 = p_n,
    \end{cases}
    \]
    where we recall from Definition \ref{def:rev} that $\Rev(I) = \{ n+1 - i : i \in I\}$.
    Suppose that $p_{n-1}=1$, so that $d_m=n-1$. By Corollary \ref{cor:zero-set-32-1-avoiding}, we have $Z(p)=D'(p),$ where $D'(p)=\{d_i+1:1\le i\le m\}.$
    Therefore
    \[
    \Rev(Z(p))
    =\{n-d_m,n-d_{m-1},\ldots,n-d_1\}=\flip(D(p)).
    \]
    Since $s = \F(p) = \descmap{3\underline{21}}^{-1}\circ \flip\circ \descmap{\underline{32}1}(s),$ we have $D(s)=\flip(D(p)).$ Hence $D(s)=\Rev(Z(p)).$
    Now suppose that $p_{n-1}=0$, so that $d_m=n-2$. By the same corollary, $Z(p)=D'(p)\cup\{n\}.$
    Removing the index of the final zero gives
    \[
    \Rev(Z(p)\setminus\{n\}) = \{n-d_m,n-d_{m-1},\ldots,n-d_1\} = \flip(D(p)).
    \]
    Again, since $s=\F(p)$, we have $D(s)=\flip(D(p)).$ Therefore $D(s)=\Rev(Z(p)\setminus\{n\}).$

    We now have
    \[
    \lmax(\tau) = \Rev(\rmin(\pi)) =  \Rev(Z(p)),
    \]
    where the first equality follows from Lemma \ref{lem:rmin_to_lmax} and the second follows from
    Lemma \ref{lem:zero_equal_rmin}.
    By Lemma \ref{lem:decrease_to_lmax} (for $\tau \in \maxset{n}{3\underline{21}}$ and its Lehmer code $t$), we can now conclude that
    \[
    D(t)= 
    \begin{cases}
    \Rev(Z(p)) &\mbox{if } t_1 > t_2, \\
    \Rev(Z(p)) \setminus \{1\} &\mbox{if } t_1 = t_2.
    \end{cases}
    \]

Suppose that $p_{n-1}=p_n=0$. We claim that for $t=\revkompl{p}$, we have $t_1 = t_2$. 
By Lemmas \ref{lem:reversal_of_lehmer} and \ref{lem:complement_of_lehmer}, 
$$
t_1 = (n-1) - ( \pi_n - 1) 
= n  - \pi_n
$$
and
$$
t_2 = (n-2) - (\pi_{n-1}-1 )
= (n-1) - \pi_{n-1}.
$$
We have $t_1=t_2$ if and only if $\pi_{n} = \pi_{n-1} + 1$, this is true by Lemma \ref{lem:zero_set_entry_incresaing_by_1}.

In the other case, we have $p_{n-1}=1$ and $p_n=0$. By Lemmas \ref{lem:reversal_of_lehmer} and \ref{lem:complement_of_lehmer}, we then have
$$
t_1 = n-\pi_n
$$
and
$$
t_2 = n - \pi_{n-1},
$$
and therefore $t_1 \neq t_2$ because $\pi_n \neq \pi_{n-1}$.
In conclusion, we have $p_{n-1} = 0$ if and only if $t_1 = t_2$.

    Moreover,
    $\Rev(\{n\}) = \{1\}.$
    Thus the final zero removed from $Z(p)$ in the case $p_{n-1} = p_n$ corresponds exactly to the position $1$ removed from $\Rev(Z(r))$ in the case $t_1 = t_2$.
    
    Comparing the two case descriptions of $D(s)$ and $D(t)$, we obtain
    $D(s) = D(t).$
    Consequently, by Lemma~\ref{lem:delta_3-21_bijection} 
    we have $\F(p) = s = t = \revkompl{p}$.
\end{proof}